\documentclass[12pt]{amsart}
\usepackage[T1]{fontenc}
\usepackage{lmodern}
\usepackage{amsmath,amssymb,amsthm,mathtools}
\usepackage[colorlinks,linkcolor=blue,citecolor=blue,urlcolor=blue]{hyperref}
\newtheorem{theorem}{Theorem}[section]
\newtheorem{proposition}[theorem]{Proposition}
\newtheorem{lemma}[theorem]{Lemma}
\numberwithin{equation}{section}
\allowdisplaybreaks
\title{On the Gauss sphere problem}
\author{Chiara Bellotti}
\address{Max Planck Institute for Mathematics, Bonn, Germany}
\email{bellotti@mpim-bonn.mpg.de}
\author{Bryce Kerr}
\address{School of Science, University of New South Wales, Canberra, ACT 2106,
Australia}
\email{ bryce.kerr@unsw.edu.au}

\begin{document}
\begin{abstract}
In this article we improve Heath-Brown's estimate for the error term in the three-dimensional sphere problem, proving
\[
\#\{\mathbf x\in\mathbb Z^3:\|\mathbf x\|\le R\}=\frac{4\pi}{3}R^3+O_\varepsilon(R^{55/42+\varepsilon}).
\]
Incoporating the recent OpenAI quasi-GRH into our argument further improves this to
\[
\#\{\mathbf x\in\mathbb Z^3:\|\mathbf x\|\le R\}=
\frac{4\pi}{3}R^3+O_\varepsilon(R^{\frac{2274811}{1743872}+\varepsilon}).
\]
\end{abstract}
\maketitle

\section{Introduction}
Let
\[
 S(R)=\#\{\mathbf x\in\mathbb Z^3:\|\mathbf x\|\le R\},
\]
where $\|\mathbf x\|$ is the Euclidean norm. An observation dating back to Gauss is that the number of lattice points in a convex body is asymptotic to its volume, thus we have 
\begin{align}
\label{eq:SR}
S(R)= \frac{4\pi}{3}R^3+o(R^3).
\end{align}
It's expected that 
$$S(R)=\frac{4\pi}{3}R^3+O(R^{1+o(1)})$$
see for example~\cite{ChamizoPastor}while it's known that GRH implies 
$$S(R)=\frac{4\pi}{3}R^3+O(R^{5/4+o(1)})$$
see for example~\cite{HeathBrown}. The question of how far the error in~\eqref{eq:SR} can be reduced has a long
history.  Chen~\cite{Chen} and Vinogradov~\cite{Vinogradov} independently
obtained an error of $O(R^{4/3})$. Chamizo and Iwaniec~\cite{Cham} later
obtained the exponent $29/22$, and Heath-Brown~\cite{HeathBrown} improved
this to $21/16$. We prove the following estimate, which improves the latter
exponent.

\begin{theorem}\label{thm:main}
For every $\varepsilon>0$,
\[
 S(R)=\frac{4\pi}{3}R^3+O_\varepsilon(R^{55/42+\varepsilon}).
\]
\end{theorem}

Note that 
$$\frac{21}{16}=\frac{55}{42}+\frac{1}{336}.$$
Incoporating the recent $7/8$-quasi-GRH result announced by OpenAI~\cite{OpenAIQRH},
into our argument, we obtain a further improvement.
\begin{theorem}\label{thm:qrh}
For every
$\varepsilon>0$,
\[
 S(R)=\frac{4\pi}{3}R^3+O_\varepsilon(R^{\frac{2274811}{1743872}+\varepsilon}).
\]
\end{theorem}
Our proof of Theorem \ref{thm:main} uses a reduction of Chamizo and Iwaniec
\cite{Cham} to class numbers in short intervals and hence bilinear forms with real characters of the type
\begin{align*}
\sum_{K\le n \le K+N}\sum_{m\sim M}\left(\frac{n}{m}\right)
\end{align*}
with $(n/m)$ denoting the Legendre symbol. One should think of $N$ being much smaller than $K$ so in this scenario a direct application of Heath-Brown's quadratic large sieve~\cite{HBsieve} is not effective due to small summation length compared to conductor. Variations of~~\cite{HBsieve} were used by Heath-Brown in his work on the sphere problem~\cite{HeathBrown} and our main novelty is to improve on~\cite[Lemma~5]{HeathBrown}, see Section~\ref{sec:4th} below, by first performing some manipulations so that a reduction to the quadratic large sieve involves counting how often the product of two values of a quadratic polynomial are a square. A similar problem was considered by Chinis and Shala~\cite{ChinisShala} in their work on the random Chowla conjecture and our approach follows along these lines  via reduction to  estimates of Cilleruelo and Garaev~\cite{CillGaraev}.

Along the way, we make use of other progress in analytic number theory, including Bourgain's exponent pair~\cite{Bourgain} and Conrey and Iwaniec's~\cite{CoIw} Weyl strength subconvexity estimate for real characters (which we use in the form as stated by Petrow and Young~\cite{PetrowYoung}).

\section{The smoothing decomposition}\label{sec:decomposition}
We may assume $R\ge2$. Throughout the proof, put
\begin{equation}\label{eq:H}
 H=R^{-13/21}.
\end{equation}
Choose a fixed nonincreasing smooth function $\rho_0$ on $\mathbb R$
with $0\le\rho_0\le1$, $\rho_0(u)=1$ for $u\le0$, and $\rho_0(u)=0$
for $u\ge1$. Set
\[
 f(r)=\rho_0\left(\frac{r-R}{H}\right)\qquad(r\ge0).
\]
Then $f=1$ on $[0,R]$, $f=0$ on $[R+H,\infty)$, and
$f^{(j)}\ll_j H^{-j}$ for $j\ge1$. 

Let $r_3(n)$ be the number of representations of $n$ as a sum of
three integer squares, with signs and order counted.  Define
\[
 S^{(H)}(R)=\sum_{\mathbf x\in\mathbb Z^3}f(\|\mathbf x\|),
 \qquad
 B^{(H)}(R)=\sum_{R^2<n\le(R+H)^2}r_3(n)f(\sqrt n).
\]
We have the exact identity
\begin{equation}\label{eq:decomposition}
 S(R)=S^{(H)}(R)-B^{(H)}(R).
\end{equation}
Thus, Theorem~\ref{thm:main} follows from the following two propositions.

\begin{proposition}\label{prop:arithmetic}
With $H$ as in \eqref{eq:H},
\[
 B^{(H)}(R)=4\pi\int_R^{R+H}r^2f(r)\,dr
              +O_\varepsilon(R^{55/42+\varepsilon}).
\]
\end{proposition}
\begin{proposition}\label{prop:analytic}
With $H$ as in \eqref{eq:H},
\[
 S^{(H)}(R)=4\pi\int_0^\infty r^2f(r)\,dr
              +O_\varepsilon(R^{55/42+\varepsilon}).
\]
\end{proposition}
We prove Proposition~\ref{prop:arithmetic} in Section~\ref{sec:arithmetic}
and Proposition~\ref{prop:analytic} in Section~\ref{sec:analytic}.

\section{The arithmetic boundary term}\label{sec:arithmetic}
\subsection{Truncation and separation of the squares}
This section follows some of the manipulations of Chamizo and Iwaniec~\cite[pp. 418 and 426]{Cham}. Let $R_3(n)$ count the primitive representations of $n$ as a sum of three squares:
\[
R_3(n)=\#\{(x_1,x_2,x_3)\in\mathbb Z^3:x_1^2+x_2^2+x_3^2=n,\ \gcd(x_1,x_2,x_3)=1\}.
\] Then
\[
 r_3(N)=\sum_{d^2\mid N}R_3(N/d^2).
\]
By~\cite[(1.1)--(1.2)]{Cham}, for $n\equiv\nu\pmod8$, the class-number formula gives
\begin{equation}\label{eq:class-number}
 R_3(n)=\frac{c_\nu}{\pi}\sqrt n\,L(1,\chi_n),
 \qquad \chi_n(m)=\left(\frac{-4n}{m}\right),
\end{equation}
where $(-4n/m)$ denotes the the Kronecker symbol and
\[
 c_1=c_2=c_5=c_6=24,\qquad c_3=16,\qquad c_0=c_4=c_7=0.
\]

Thus
\begin{equation}\label{eq:boundary-class}
 B^{(H)}(R)=\frac1\pi\sum_{d\le R+H}\sum_{\nu\bmod8}c_\nu
 \sum_{\substack{R^2/d^2<n\le(R+H)^2/d^2\\n\equiv\nu\;(8)}}
 \sqrt n\,f(d\sqrt n)L(1,\chi_n).
\end{equation}
Every character $\chi_n$ here is nonprincipal: its inducing primitive
character has negative fundamental discriminant.

Choose a fixed smooth function $U$ on $[0,\infty)$ that is one on
$[0,1]$, zero on $[2,\infty)$, and satisfies $0\le U\le1$.
For a nonprincipal character modulo $q$, partial summation and the
P\'olya--Vinogradov inequality give
\begin{equation}\label{eq:truncation}
 L(1,\chi)=\sum_{m\ge1}\frac{\chi(m)}m U(m/Y)
       +O\left(\frac{\sqrt q\log(2q)}Y\right)\qquad(Y\ge q).
\end{equation}
Indeed, the partial sums of $\chi$ are $O(\sqrt q\log(2q))$, and
\[
 \int_Y^\infty\left|\frac{d}{dt}
       \frac{1-U(t/Y)}t\right|\,dt\ll Y^{-1}.
\]
For each $d$, take $Z_d=16R^2/d^2$. On the interval in
\eqref{eq:boundary-class}, we have $4n\le Z_d$ and $n\asymp R^2/d^2$.
After multiplication by $\sqrt n$, the error in \eqref{eq:truncation}
is therefore $O(\log(2R))$ per pair $(d,n)$. There are
\begin{equation}\label{eq:pair-count}
 \sum_{d\le R+H}O\left(1+\frac{RH}{d^2}\right)=O(R)
\end{equation}
such pairs. Thus the error induced in~\eqref{eq:boundary-class} by truncating $L(1,\chi_n)$ in~\eqref{eq:truncation} is $O(R\log(2R))$.

When $m=a^2$, the character value is $1$ if $(a,4n)=1$ and zero
otherwise. We may replace $U(a^2/Z_d)$ by one in these terms: the
error per pair $(d,n)$ is at most
\[
 \sqrt n\sum_{a>\sqrt{Z_d}}a^{-2}\ll1,
\]
and \eqref{eq:pair-count} bounds the total error by $O(R)$.
It follows that
\begin{equation}\label{eq:boundary-split}
 B^{(H)}(R)=M(R)+E(R)+O(R\log(2R)),
\end{equation}
where
\begin{align}
 M(R)&=\frac1\pi\sum_{d\le R+H}\sum_{\nu\bmod8}c_\nu
 \sum_{\substack{R^2/d^2<n\le(R+H)^2/d^2\\n\equiv\nu\;(8)}}
 \sqrt n\,f(d\sqrt n)\sum_{(a,4n)=1}\frac1{a^2},
 \label{eq:M}\\
 E(R)&=\frac1\pi\sum_{d\le R+H}\sum_{\nu\bmod8}c_\nu
 \sum_{\substack{R^2/d^2<n\le(R+H)^2/d^2\\n\equiv\nu\;(8)}}
 \sqrt n\,f(d\sqrt n)\notag\\
 &\hspace{18mm}\times\sum_{m\ne\square}\frac{\chi_n(m)}m U(m/Z_d).
 \label{eq:E}
\end{align}

\subsection{The square contribution}
For an odd positive integer $a$ and a continuously differentiable
function $g$ on $[X,Y]$, elementary summation in arithmetic progressions
gives
\begin{equation}\label{eq:coprime-ap}
 \sum_{\substack{X<n\le Y\\n\equiv\nu\;(8)\\(n,a)=1}}g(n)
 =\frac{\varphi(a)}{8a}\int_X^Y g(x)\,dx
 +O\left(\tau(a)\left(\|g\|_\infty+\int_X^Y|g'(x)|\,dx\right)\right).
\end{equation}
Here $\varphi$ is Euler's totient function and $\tau$ counts positive
divisors. To see this, expand $1_{(n,a)=1}=\sum_{k\mid(n,a)}\mu(k)$.
For each $k\mid a$, the two congruences specify one residue class
modulo $8k$. Partial summation gives the integral with coefficient
$1/(8k)$ and error $O(\|g\|_\infty+\int|g'|)$. Summing over $k$
proves \eqref{eq:coprime-ap}.

Apply this with $g(x)=\sqrt x f(d\sqrt x)$ on the interval in
\eqref{eq:M}. The derivative is
\[
 g'(x)=\frac{f(d\sqrt x)}{2\sqrt x}+\frac d2 f'(d\sqrt x).
\]
The substitution $r=d\sqrt x$ and the bounds for $f$ show that
\[
 \|g\|_\infty+\int|g'(x)|\,dx\ll R/d,
 \qquad
 \int g(x)\,dx=\frac2{d^3}\int_R^{R+H}r^2f(r)\,dr.
\]
Since $\sum_a\tau(a)a^{-2}$ converges, we can sum the errors in
\eqref{eq:coprime-ap} with weights $a^{-2}$. The resulting error for
fixed $d,\nu$ is $O(R/d)$. The coefficient of the integral is determined
by
\[
 \sum_{\substack{a\ge1\\a\text{ odd}}}\frac{\varphi(a)}{a^3}
 =\prod_{p\ne2}\frac{1-p^{-3}}{1-p^{-2}}
 =\frac{6\zeta(2)}{7\zeta(3)}.
\]
Consequently, the inner sum in \eqref{eq:M} equals
\[
 \frac{3\zeta(2)}{14\zeta(3)d^3}
       \int_R^{R+H}r^2f(r)\,dr+O(R/d).
\]
Now $\sum_{\nu\bmod8}c_\nu=112$ and
$\sum_{d\le R+H}d^{-3}=\zeta(3)+O(R^{-2})$. The error from extending
this last sum is $O(1)$, because the integral is $O(R^2H)$. Using
$\zeta(2)=\pi^2/6$ gives
\begin{equation}\label{eq:squares-evaluated}
 M(R)=4\pi\int_R^{R+H}r^2f(r)\,dr+O(R\log(2R)).
\end{equation}

\subsection{Decomposition of the nonsquare contribution}
We next introduce a smooth partition of unity for summation over $n$ in~\eqref{eq:E}.
Choose a fixed nonincreasing smooth function $\vartheta$ that is one
on $[0,4/3]$, zero on $[5/3,\infty)$, and between zero and one.
Put $\Psi(u)=\vartheta(u)-\vartheta(2u)$ for $u>0$, extended by zero
for $u\le0$. Then $\Psi$ is smooth, nonnegative, and supported in
$[2/3,5/3]$. Telescoping gives
\[
 1=\vartheta(x)+\sum_{\substack{L=2^\ell\\\ell\ge1}}\Psi(x/L)
 \qquad(x>0).
\]
Use this partition with $x=n-R^2/d^2$. The term with $\vartheta$
contains only $O(1)$ integers $n$ for each $d$. Since $\sqrt n\ll R/d$
and $\sum_{m\le2Z_d}m^{-1}\ll\log(2R)$, its total contribution to
$E(R)$ is $O(R\log^2(2R))$.

For the other terms, fix a nonnegative smooth function $w$ supported
in $(1/2,2)$ such that $\sum_{j\in\mathbb Z}w(2^{-j}x)=1$ for $x>0$.
For an integer $m\ge1$, the terms with $j<0$ vanish. Insert this
partition in the $m$-sum. For $J=2^j$, $j\ge0$, and $L=2^\ell$,
$\ell\ge1$, define
\begin{equation}\label{eq:block}
 \begin{split}
 E_\nu(d;L,J)={}&
 \sum_{\substack{n>R^2/d^2\\n\equiv\nu\;(8)}}
 \sqrt n\,f(d\sqrt n)\Psi\left(\frac{n-R^2/d^2}{L}\right)\\
 &\hspace{12mm}\times\sum_{m\ne\square}\frac{\chi_n(m)}m
                   w(m/J)U(m/Z_d).
 \end{split}
\end{equation}
Thus
\begin{equation}\label{eq:E-blocks}
 E(R)=\frac1\pi\sum_{d\le R+H}\sum_{\nu\bmod8}c_\nu
              \sum_{L,J}E_\nu(d;L,J)+O(R\log^2(2R)).
\end{equation}
The support conditions imply
\begin{equation}\label{eq:block-ranges}
 J\ll R^2/d^2,\qquad L\ll RH/d^2,\qquad n\asymp R^2/d^2.
\end{equation}
In particular, there are $O(\log(2R))$ choices for each dyadic parameter.

We record the normalization used for both estimates. For fixed $d$,
write $Q=R/d$ and set
\[
 \begin{aligned}
 \Phi(u)&=Q^{-1}\sqrt{Q^2+Lu}\,f(d\sqrt{Q^2+Lu})\Psi(u),\\
 W(x)&=x^{-1}w(x)U(Jx/Z_d).
 \end{aligned}
\]
The first function is defined on the support of $\Psi$ and extended
by zero; the second is similarly extended outside $x>0$.
They are supported in fixed compact subintervals of $(0,\infty)$
and satisfy $\Phi^{(r)},W^{(r)}\ll_r1$. For $W$, this follows from
$J\ll Z_d$. For $\Phi$, put $a=L/Q^2\ll H/R$ and write
\[
 \Phi(u)=\sqrt{1+au}\,
 \rho_0\left(\frac RH(\sqrt{1+au}-1)\right)\Psi(u).
\]
Every positive-order derivative of the argument of $\rho_0$ is
$O_r((R/H)a^r)=O_r(1)$, proving the assertion.

For weights of this kind, define
\begin{equation}\label{eq:normalized-sum}
 S_\nu(N;L,J)=\frac1J\sum_{m\ne\square}W(m/J)
       \sum_{\substack{n\ge1\\n\equiv\nu\;(8)}}
       \Phi\left(\frac{n-N}{L}\right)\left(\frac{-4n}{m}\right).
\end{equation}
We have exactly
\begin{equation}\label{eq:normalization}
 E_\nu(d;L,J)=Q S_\nu(Q^2;L,J).
\end{equation}
The notation $S_\nu$ leaves the two weights implicit; all the estimates
below are uniform under the support and derivative bounds just proved.

\subsection{A fourth-moment estimate}
\label{sec:4th}

Our main novelty in proving Theorem~\ref{thm:main} is Proposition~\ref{prop:fourth} and improves on~\cite[Lemma~5]{HeathBrown}. Our next result will be used in the proof of Proposition~\ref{prop:fourth}.
\begin{lemma}\label{lem:fourth-moment}
For $A,T\ge1$,
\[
 \sum_{\substack{A\le b<2A\\b\text{ odd and squarefree}}}
 \sup_{\theta\in\mathbb R}
 \left|\sum_{T\le t<2T}\left(\frac tb\right)e(\theta t)\right|
 \ll_\varepsilon (AT)^\varepsilon
          \left(AT^{3/4}+A^{3/4}T^{5/4}\right).
\]

\end{lemma}
\begin{proof}
Let 
$$S=\sum_{\substack{A\le b<2A\\b\text{ odd and squarefree}}}\sup_{\theta\in\mathbb R}S_b(\theta)$$
where 
\begin{align*}
S_b(\theta)=\left|\sum_{T\le t<2T}\left(\frac tb\right)e(\theta t)\right|.\end{align*}
 On squaring and grouping the
terms according to $h=t_1+t_2$, we obtain
\[
 S_b(\theta)^2=\sum_{2T\le h<4T}e(\theta h)U_{b,h},\]
where
 \[U_{b,h}=\sum_{\substack{T\le t<2T\\T\le h-t<2T}}
                  \left(\frac{t(h-t)}b\right).
\]
Cauchy and the triangle inequality give, uniformly in $\theta$,
\begin{equation}\label{eq:fourth-cauchy}
 |S_b(\theta)|^4\ll T\sum_{2T\le h<4T}|U_{b,h}|^2.
\end{equation}
Thus by H\"{o}lder
\begin{align*}
 S^4=\left(\sum_{\substack{A\le b<2A\\b\text{ odd and squarefree}}}\sup_\theta|S_b(\theta)|\right)^4
 &\ll A^3\sum_{\substack{A\le b<2A\\b\text{ odd and squarefree}}}\sup_\theta|S_b(\theta)|^4 \\ 
&\ll A^3T\sum_{2T\le h<4T}\sum_{\substack{A\le b<2A\\b\text{ odd and squarefree}}}|U_{b,h}|^2.
\end{align*}

We use the following version of Heath--Brown's quadratic
large sieve \cite[Corollary~2]{HBsieve}:
\begin{equation}\label{eq:quadratic-large-sieve}
 \sum_{\substack{A\le b<2A\\b\text{ odd and squarefree}}}
 \left|\sum_{n\le X}a_n\left(\frac nb\right)\right|^2
 \ll_\varepsilon (AX)^\varepsilon(A+X)
       \sum_{\substack{n_1,n_2\le X\\n_1n_2=\square}}|a_{n_1}a_{n_2}|.
\end{equation}
The above implies that 
\begin{align*}
&\sum_{\substack{A\le b<2A\\b\text{ odd and squarefree}}}|U_{b,h}|^2 \\ &\ll  (A+T^2)\#\{(t_1,t_2): T\le t_i<2T,
T\le h-t_i<2T,\ t_1(h-t_1)t_2(h-t_2)=\square\}.
\end{align*}
Thus, it is enough to show uniformly for $2T\le h<4T$ that 
\[\#\{(t_1,t_2): T\le t_i<2T,\ T\le h-t_i<2T,\ t_1(h-t_1)t_2(h-t_2)=\square\}\ll T^{1+o(1)}.\]

For each positive squarefree integer $\delta$, let $A_\delta$ count
the integers $t$ satisfying $T\le t<2T$ and $T\le h-t<2T$ for which
$t(h-t)=\delta v^2$ for some $v\ge1$. Each
positive integer $t(h-t)$ has a unique squarefree kernel, so every admissible
$t$ contributes to exactly one $A_\delta$ and 
$$\sum_\delta A_\delta\ll T.$$
The product of two such values is a square precisely when their squarefree
kernels agree. Hence the pair count above is
\[
\#\{(t_1,t_2): T\le t_i<2T,\ T\le h-t_i<2T,\ t_1(h-t_1)t_2(h-t_2)=\square\}\ll \sum_\delta A_\delta^2.
\]
Moreover, if $t$ is counted by $A_\delta$ then
\[
 (2t-h)^2+4\delta v^2=h^2.
\]
Here $\delta\ll T^2$, $h\ll T$, and $|2t-h|,v\ll T$.
The discriminant of this binary quadratic equation is $-16\delta<0$.
By a result of Cilleruelo and Garaev
\cite[Proposition~1]{CillGaraev}, we get
$A_\delta\ll_\varepsilon T^\varepsilon$; 
Thus
\[
 \sum_\delta A_\delta^2
 \le \max_\delta A_\delta\sum_\delta A_\delta
 \ll_\varepsilon T^{1+\varepsilon}.
\]
It follows from \eqref{eq:quadratic-large-sieve} that
\[
 \sum_b|U_{b,h}|^2
 \ll_\varepsilon (AT)^\varepsilon(A+T^2)T.
\]
Summing over the $O(T)$ choices of $h$ in \eqref{eq:fourth-cauchy},
and then using H\"older's inequality in $b$, gives
\[
 \left(\sum_b\sup_\theta|S_b(\theta)|\right)^4
 \ll A^3\sum_b\sup_\theta|S_b(\theta)|^4
 \ll_\varepsilon (AT)^\varepsilon(A^4T^3+A^3T^5).
\]
This completes the proof.
\end{proof}

\begin{proposition}\label{prop:fourth}
For the sum in \eqref{eq:normalized-sum},
\begin{equation}\label{eq:fourth-bound}
 S_\nu(N;L,J)\ll_\varepsilon R^\varepsilon
                   \left((JL)^{1/4}+J^{1/2}L^{-1/4}\right),
\end{equation}
uniformly when $N,L,J\ge1$ are bounded by fixed powers of $R$ and
the weights satisfy the support and derivative bounds stated above.
\end{proposition}
\begin{proof}
Recall that 
$$
 S_\nu(N;L,J)=\frac1J\sum_{m\ne\square}W(m/J)
       \sum_{\substack{n\ge1\\n\equiv\nu\;(8)}}
       \Phi\left(\frac{n-N}{L}\right)\left(\frac{-4n}{m}\right).
$$
The argument is to apply Poission summation in $n$ then Lemma~\ref{lem:fourth-moment}.

 Only odd $m$ contribute. Write $m=bc^2$, where $b>1$ is odd and
squarefree and $c$ is odd. The identity
\[
 \left(\frac{-4n}{bc^2}\right)
 =\left(\frac{-4n}{b}\right)1_{(n,c)=1}
\]
and M\"obius inversion reduce the $n$-sum to terms with $e\mid c$,
$(e,b)=1$, and $n=er$. Terms with $(e,b)>1$ vanish. For such a pair
$c,e$, the sum over $r$ is
\[
 \sum_{r\equiv r_0\;(8)}\left(\frac rb\right)
                 \Phi\left(\frac{er-N}{L}\right),
 \qquad r_0\equiv\nu e^{-1}\pmod8.
\]
The sum can be taken over all integers because the support of $\Phi$
forces $er>N>0$. Let
$\tau_b=\sum_{v\bmod b}(v/b)e(v/b)$ be the quadratic Gauss sum.
Poisson summation modulo $8b$ gives
\begin{equation}\label{eq:poisson-character}
 \begin{split}
 &\sum_{r\equiv r_0\;(8)}\left(\frac rb\right)
                 \Phi\left(\frac{er-N}{L}\right)\\
 &\qquad=\frac{L\tau_b}{8eb}\left(\frac8b\right)
     \sum_{t\in\mathbb Z}\left(\frac tb\right)
       e\left(\frac{br_0t}{8}-\frac{Nt}{8eb}\right)
       \widehat\Phi\left(\frac{Lt}{8eb}\right).
 \end{split}
\end{equation}
For completeness, the finite Fourier sum here is
\[
 \sum_{\substack{a\bmod8b\\a\equiv r_0\;(8)}}
          \left(\frac ab\right)e\left(\frac{ta}{8b}\right)
 =\left(\frac8b\right)e\left(\frac{br_0t}{8}\right)
       \tau_b\left(\frac tb\right).
\]
This follows by writing $a=r_0+8u$ and then changing from $u$ to
$a\bmod b$. We used $b^{-1}\equiv b\pmod8$. Since $b>1$, the term
$t=0$ is zero; also $|\tau_b|=\sqrt b$.

Let $C_{c,e}$ denote the contribution for this fixed $c,e$ to
$S_\nu(N;L,J)$. Substituting \eqref{eq:poisson-character} gives
\begin{equation}\label{eq:Cce}
 \begin{split}
 C_{c,e}={}&\frac{L\mu(e)}{8Je}
 \sum_{\substack{b>1\text{ odd and squarefree}\\(b,e)=1}}
 \frac{W(bc^2/J)\tau_b}{b}
       \left(\frac{-4e}{b}\right)\left(\frac8b\right)\\
 &\qquad\times\sum_{t\ne0}\left(\frac tb\right)
 e\left(\frac{br_0t}{8}-\frac{Nt}{8eb}\right)
 \widehat\Phi\left(\frac{Lt}{8eb}\right).
 \end{split}
\end{equation}
Put $A=J/c^2$ and $T_0=Ae/L$. The support of $W$ implies
$A/2<b<2A$, which is covered by a fixed number of intervals of the
form used in Lemma~\ref{lem:fourth-moment}. Nonempty sums have
$A\gg1$. Splitting positive and negative $t$ and replacing $t$ by
$-t$ in the latter reduces the estimation to positive $t$; the extra
factor $(-1/b)$ has absolute value one.

Divide these positive frequencies into intervals $T\le t<2T$, where
$T=1,2,4,\ldots$. We explain how the smooth Fourier weight can be
removed uniformly in $b$. Choose a fixed smooth function $\chi$
supported in $(1/2,3)$ and equal to one on $[1,2]$. For either sign,
put
\[
 g_b(u)=\chi(u)\widehat\Phi\left(\frac{\pm LTu}{8eb}\right).
\]
Rapid decay of $\widehat\Phi$ and its derivatives, followed by two
integrations by parts in its Fourier transform, gives, for every $B>0$,
\begin{equation}\label{eq:weight-envelope}
 |\widehat g_b(v)|\ll_B
       (1+T/T_0)^{-B}(1+|v|)^{-2}.
\end{equation}
The bound is uniform in $b$. Indeed, on the support of $\chi$ the
argument of $\widehat\Phi$ has size comparable with $T/T_0$; the
factors introduced by differentiating it are absorbed by choosing
sufficiently many powers in its rapid-decay bound.

Fourier inversion expresses $g_b(t/T)$ as an integral of
$e(vt/T)\widehat g_b(v)$. All the other factors depending on $t$ in
\eqref{eq:Cce} form an additive phase. The integrable bound
\eqref{eq:weight-envelope} and Lemma~\ref{lem:fourth-moment} therefore
imply
\begin{equation}\label{eq:dyadic-Cce}
 |C_{c,e}|\ll_\varepsilon
 \frac{L}{Je\sqrt A}
 \sum_{T=1,2,4,\ldots}(1+T/T_0)^{-B}
 (AT)^\varepsilon\left(AT^{3/4}+A^{3/4}T^{5/4}\right).
\end{equation}
No common additive phase for different $b$ is required, because the
supremum in Lemma~\ref{lem:fourth-moment} is taken separately for
each modulus.

For $a>0$ and $B>a$,
\[
 \sum_{T=1,2,4,\ldots}T^a(1+T/T_0)^{-B}\ll_{a,B}T_0^a
 \qquad(T_0>0).
\]
When $T_0\ge1$, split at $T=T_0$ and sum two geometric series.
When $T_0<1$, the left side is $O(T_0^B)\le O(T_0^a)$.
Thus the same calculation handles both short dual sums and their tails.
Taking $B$ sufficiently large in \eqref{eq:dyadic-Cce} gives
\[
 |C_{c,e}|\ll_\varepsilon R^\varepsilon
 \frac{L}{Je\sqrt A}
       \left(AT_0^{3/4}+A^{3/4}T_0^{5/4}\right).
\]
Here the small powers of $A$ and $T_0$ have been absorbed into
$R^\varepsilon$; if $T_0<1$, they only improve the estimate.
Substituting $A=J/c^2$ and $T_0=Ae/L$ yields
\[
 |C_{c,e}|\ll_\varepsilon R^\varepsilon
 \left((JL)^{1/4}c^{-5/2}e^{-1/4}
       +J^{1/2}L^{-1/4}c^{-3}e^{1/4}\right).
\]
Both sums over $c$ and $e\mid c$ converge, since
\[
 \sum_{e\mid c}e^{-1/4}\le\tau(c),\qquad
 \sum_{e\mid c}e^{1/4}\le c^{1/4}\tau(c).
\]
Summing proves \eqref{eq:fourth-bound}.
\end{proof}

\subsection{The Petrow--Young estimate}
\begin{lemma}\label{lem:PY}
For $n,J\ge1$ and the smooth weights $W$ considered above,
\[
 \sum_{m\ne\square}\chi_n(m)W(m/J)
       \ll_\varepsilon J^{1/2}n^{1/6+\varepsilon}.
\]
\end{lemma}
\begin{proof}
The character $\chi_n$ is induced by a primitive quadratic character
$\chi_n^*$ of conductor at most $4n$. More explicitly, writing
$-4n=Dv^2$ with $D<0$ a fundamental discriminant, the inducing
character is $(D/\cdot)$. In particular it is nonprincipal.
The omitted Euler factors satisfy
\[
 L(s,\chi_n)=L(s,\chi_n^*)
       \prod_{\substack{p\mid4n\\p\nmid |D|}}(1-\chi_n^*(p)p^{-s}),
 \qquad
 \prod_{p\mid4n}(1+p^{-1/2})\ll_\varepsilon n^\varepsilon.
\]
Petrow and Young's Weyl bound \cite[Theorem~1.1]{PetrowYoung} therefore
gives
\[
 L(1/2+it,\chi_n)\ll_\varepsilon
       n^{1/6+\varepsilon}(1+|t|)^{1/6+\varepsilon}.
\]
Let $\widetilde W(s)=\int_0^\infty W(x)x^{s-1}\,dx$ be the Mellin
transform. Its rapid decay on vertical lines is uniform in $W$.
Mellin inversion and a shift from $\Re s=2$ to $\Re s=1/2$ give
\[
 \sum_{m\ge1}\chi_n(m)W(m/J)
 =\frac1{2\pi i}\int_{(1/2)}L(s,\chi_n)J^s\widetilde W(s)\,ds
 \ll_\varepsilon J^{1/2}n^{1/6+\varepsilon}.
\]
There is no pole to cross. Finally, the terms with $m=a^2$ contribute
$O(J^{1/2})$, since $W$ is bounded and supported where $a^2$ is
comparable with $J$. Removing them proves the lemma.
\end{proof}
For $N=Q^2$ in \eqref{eq:normalized-sum}, the $n$-sum has $O(L)$
terms and $n\asymp Q^2$. Thus Lemma~\ref{lem:PY} gives
\begin{equation}\label{eq:PY-bound}
 S_\nu(Q^2;L,J)\ll_\varepsilon R^\varepsilon Q^{1/3}LJ^{-1/2}.
\end{equation}

\subsection{Combining the estimates}
For nonnegative $A,B,C$ and $0\le u,v\le1$,
\[
 \min(A+B,C)\le\min(A,C)+\min(B,C)
                      \le A^uC^{1-u}+B^vC^{1-v}.
\]
Apply this to \eqref{eq:fourth-bound} and \eqref{eq:PY-bound}, with
\[
 A=(JL)^{1/4},\qquad B=J^{1/2}L^{-1/4},\qquad
 C=Q^{1/3}LJ^{-1/2},
\]
and take $u=2/3$ and $v=1/2$.
The powers of $J$ cancel in both terms. Using
\eqref{eq:normalization}, we obtain the bound valid for every $J$:
\begin{equation}\label{eq:combined}
 E_\nu(d;L,J)\ll_\varepsilon R^\varepsilon
       \left(Q^{10/9}L^{1/2}+Q^{7/6}L^{3/8}\right),
 \qquad Q=R/d.
\end{equation}
Summing over $J$ costs at most $O(\log(2R))$. Both powers of $L$ are
positive, so the dyadic $L$-sum is bounded by its largest scale,
$L\ll RH/d^2$. Consequently \eqref{eq:E-blocks} gives
\[
 \begin{split}
 E(R)\ll_\varepsilon R^\varepsilon\bigg(&R
 +R^{10/9}(RH)^{1/2}\sum_{d\ge1}d^{-19/9}\\
 &+R^{7/6}(RH)^{3/8}\sum_{d\ge1}d^{-23/12}\bigg).
 \end{split}
\]
Both sums over $d$ converge. With $H=R^{-13/21}$, the two remaining
powers of $R$ are $R^{82/63}$ and $R^{55/42}$, respectively. Since
$82/63<55/42$, we have
\[
 E(R)\ll_\varepsilon R^{55/42+\varepsilon}.
\]
Together with \eqref{eq:boundary-split} and
\eqref{eq:squares-evaluated}, this proves
Proposition~\ref{prop:arithmetic}.

\section{The analytic estimate}\label{sec:analytic}
\subsection{Poisson summation}
Since $f$ is constant near zero, the function
$\mathbf x\mapsto f(\|\mathbf x\|)$ is smooth and compactly supported
on $\mathbb R^3$. Its Fourier transform at a vector of length
$\xi>0$ is
\[
 4\pi\int_0^\infty r^2f(r)
       \frac{\sin(2\pi r\xi)}{2\pi r\xi}\,dr.
\]
Poisson summation, grouping the nonzero dual vectors according to their
squared length, therefore gives
\begin{equation}\label{eq:radial-poisson}
 S^{(H)}(R)=4\pi\int_0^\infty r^2f(r)\,dr
       +2\sum_{n\ge1}\frac{r_3(n)}{\sqrt n}\mathcal I(\sqrt n),
\end{equation}
where
\[
 \mathcal I(\xi)=\int_0^\infty rf(r)\sin(2\pi r\xi)\,dr.
\]
We estimate the series by smooth dyadic sums of
$r_3(n)e(\rho\sqrt n)$, first recording the exponential-sum bounds.

\subsection{The exponential-sum estimate}
For a real smooth weight $v$ supported in $(1/2,2)$, put
\[
 V_N(\rho;v)=\sum_{n\ge1}r_3(n)v(n/N)e(\rho\sqrt n).
\]
The estimates in this subsection are uniform when each derivative of
$v$ is bounded. In ranges written $x\asymp X$, the variable lies in
an interval whose endpoints are comparable with $X$; all estimates
also hold for subintervals of those ranges.

We use the following formulation of the exponent-pair bound
\cite[Section~5]{Robert}. For a fixed exponent pair $(\kappa,\lambda)$,
there are an integer $q_0$ and a constant $\eta_0>0$ such that a phase
\[
 F(y)=\Lambda y^{2/3}+u(y),\qquad \Lambda>0,
\]
on an interval $I\subseteq[D,2D]$, satisfying
\[
 |u^{(j)}(y)|\le\eta_0\Lambda D^{2/3-j}
       \quad(1\le j\le q_0),
\]
obeys
\begin{equation}\label{eq:exponent-pair-principle}
 \sum_{y\in I\cap\mathbb Z}e(F(y))
 \ll (\Lambda D^{-1/3})^\kappa D^\lambda
                +\frac{D^{1/3}}{\Lambda}.
\end{equation}
We apply this when $\Lambda D^{2/3}\gg D$. This states precisely the
stability under small differentiated errors that is needed below.
The constants may depend on the exponent pair.

\begin{lemma}\label{lem:exponent-pair}
For $\rho\ge2$ and $1\le N\le\rho^2$,
\begin{equation}\label{eq:general-pair}
 V_N(\rho;v)\ll_\varepsilon(\rho N)^\varepsilon
 \left(N^{5/4}
 +\rho^{(1+\kappa)/4}N^{7/8+\lambda/4-\kappa/8}\right).
\end{equation}
\end{lemma}
\begin{proof}
Partial summation removes $v$ at a bounded cost, leaving incomplete
sums with $n$ in an interval contained in $[N/2,2N]$. We use the
reduction in the proof of \cite[Lemma~3.1, through (3.4)]{Cham}, which
is uniform in these endpoints. We describe the resulting sums and
verify the additional derivative information required for
\eqref{eq:exponent-pair-principle}.

Choose a small $\delta>0$ in terms of $\varepsilon$. In that reduction,
a coordinate of length $O(N^{1/2})$ is split into intervals of length
$N^{1/2-\delta}$ before Cauchy's inequality is applied. The resulting
difference of two squared coordinates has size $D\ll N^{1-\delta}$.
The losses $N^{O(\delta)}$, as well as the divisor and logarithmic
factors, are absorbed into $(\rho N)^\varepsilon$. The range
$D\ll N^{3/2}/\rho$ is handled by the first-derivative estimate and
contributes $O_\varepsilon(N^{5/4+\varepsilon})$ to $V_N$.

For the remaining range, the one-dimensional Poisson and
stationary-phase transformation, followed by Cauchy's inequality,
gives parameters
\begin{equation}\label{eq:analytic-parameters}
 N^{3/2}/\rho\ll D\ll N^{1-\delta},\qquad
 U=\rho DN^{-3/2},\qquad 1\le T\le U,
\end{equation}
and the bound
\begin{equation}\label{eq:CI-reduction}
 |V_N(\rho;v)|^4\ll_\varepsilon(\rho N)^\varepsilon
 \left(N^5+\rho^{-1}N^{9/2}
 \sum_{x\asymp U}\sum_{z\asymp T}
       \left|\sum_{y\asymp D}e(G(x,y,z))\right|\right).
\end{equation}
The error $\rho^{-1}DN^{9/2}$ in the cited reduction has been absorbed
into $N^5$, using $D\le N$ and $\rho\ge\sqrt N$.

Here is the phase in \eqref{eq:CI-reduction}. Put
\[
 \phi(t,y)=\rho(\sqrt t-\sqrt{t+y}),
\]
and let $\alpha(s,y)>0$ be its stationary point, defined by
\begin{equation}\label{eq:stationary}
 \phi_t(\alpha,y)=s,
 \quad\text{or equivalently}\quad
 \alpha^{-1/2}-(\alpha+y)^{-1/2}=2s/\rho.
\end{equation}
The left side is strictly decreasing from infinity to zero as
$\alpha$ increases, so there is exactly one positive solution.
Define
\[
 g(s,y)=\phi(\alpha(s,y),y)-s\alpha(s,y),\qquad
 G(x,y,z)=g(x+z,y)-g(x,y).
\]
Differentiating $g$ with respect to $s$ and using
\eqref{eq:stationary} cancels the terms containing $\alpha_s$ and
gives $g_s=-\alpha$. Hence
\begin{equation}\label{eq:G-integral}
 -G(x,y,z)=\int_x^{x+z}\alpha(s,y)\,ds.
\end{equation}

We now estimate $\alpha$ and its derivatives. For fixed $s$, write
\[
 a_0=\left(\frac{\rho y}{4s}\right)^{2/3},\qquad
 \eta=\frac{y}{a_0},\qquad \alpha=a_0A.
\]
Using the exact identity
\[
 \alpha^{-1/2}-(\alpha+y)^{-1/2}
       =\frac y2\int_0^1(\alpha+u y)^{-3/2}\,du,
\]
the stationary equation becomes
\begin{equation}\label{eq:normalized-stationary}
 \int_0^1(A+u\eta)^{-3/2}\,du=1.
\end{equation}
For $s\asymp U$ and $y\asymp D$, we have $a_0\asymp N$ and
$\eta\asymp D/N\ll N^{-\delta}$. We may therefore assume
$0<\eta<1/2$, since bounded $N$ is covered trivially.

The integral in \eqref{eq:normalized-stationary} is decreasing in
$A$. Comparing $A=1-\eta$ and $A=1$ gives
$1-\eta\le A\le1$. Differentiation yields
\[
 A'(\eta)=-\frac{\int_0^1u(A+u\eta)^{-5/2}\,du}
                     {\int_0^1(A+u\eta)^{-5/2}\,du},
\]
so $|A'|\le1$. Repeated differentiation gives
$A^{(j)}(\eta)\ll_j1$ for each fixed $j$: at every step the coefficient
of the highest derivative is a nonzero constant times
$\int_0^1(A+u\eta)^{-5/2}\,du$, which is bounded away from zero,
and all other factors are bounded by induction. These differentiations
can also be justified directly by difference quotients in
\eqref{eq:normalized-stationary}, since its derivative in $A$ is
bounded away from zero.

Since $\eta=(4s/\rho)^{2/3}y^{1/3}$, the chain rule and
$A-1=O(\eta)$ now give, for every fixed $j\ge0$,
\begin{equation}\label{eq:alpha-derivatives}
 \partial_y^j(\alpha-a_0)
 \ll_j\frac DN\left(\frac{\rho}{s}\right)^{2/3}D^{2/3-j}.
\end{equation}
Substituting into \eqref{eq:G-integral}, we obtain
\[
 -G(x,y,z)=\Lambda_{x,z}y^{2/3}+u_{x,z}(y),
 \qquad
 \Lambda_{x,z}=\int_x^{x+z}\left(\frac{\rho}{4s}\right)^{2/3}\,ds,
\]
with
\begin{equation}\label{eq:phase-derivatives}
 |u_{x,z}^{(j)}(y)|\ll_j N^{-\delta}
                      \Lambda_{x,z}D^{2/3-j}.
\end{equation}
The size of the leading phase is particularly simple:
\begin{equation}\label{eq:phase-size}
 \Lambda_{x,z}D^{2/3}
 \asymp T\rho^{2/3}U^{-2/3}D^{2/3}=TN.
\end{equation}
For sufficiently large $N$, \eqref{eq:phase-derivatives} supplies all
the finitely many derivative bounds needed in
\eqref{eq:exponent-pair-principle}. Splitting into a fixed number of
intervals contained in $[D,2D]$, with comparable scales if needed,
we conclude that
\begin{equation}\label{eq:y-sum}
 \left|\sum_{y\asymp D}e(G(x,y,z))\right|
 \ll (TN)^\kappa D^{\lambda-\kappa}+\frac{D}{TN}.
\end{equation}
Replacing $G$ by $-G$ does not change the absolute value.

There are $O(UT)$ pairs $(x,z)$. Since $T\le U$ and $D\le N$,
\[
 \begin{split}
 UT(TN)^\kappa D^{\lambda-\kappa}
 &\le U^{2+\kappa}N^\kappa D^{\lambda-\kappa}
   =\rho^{2+\kappa}D^{2+\lambda}N^{-3-\kappa/2}\\
 &\le\rho^{2+\kappa}N^{-1+\lambda-\kappa/2},
 \end{split}
\]
Also,
\[
 UT\frac{D}{TN}=\rho D^2N^{-5/2}\le\rho N^{-1/2}.
\]
Their contributions to \eqref{eq:CI-reduction} are respectively
$\rho^{1+\kappa}N^{7/2+\lambda-\kappa/2}$ and $O(N^4)$.
Taking fourth roots proves \eqref{eq:general-pair}.
\end{proof}

Bourgain \cite[Theorem~6]{Bourgain} proves that
$(13/84+\eta,55/84+\eta)$ is an exponent pair for every $\eta>0$.
Choosing $\eta$ sufficiently small in Lemma~\ref{lem:exponent-pair}
gives
\begin{equation}\label{eq:Bourgain}
 V_N(\rho;v)\ll_\varepsilon(\rho N)^\varepsilon
       \left(N^{5/4}+\rho^{97/336}N^{685/672}\right)
       \qquad(1\le N\le\rho^2).
\end{equation}
We also use the second-derivative estimate from
\cite[Lemma~3.1]{Cham}. By partial summation it applies to the same
smooth weights and gives, for $\rho\asymp R$,
\begin{equation}\label{eq:CI-second}
 V_N(\rho;v)\ll_\varepsilon(RN)^\varepsilon
       \left(N^{5/4}+R^{3/8}N^{15/16}+R^{1/8}N^{17/16}\right).
\end{equation}
For $1\le N\le R^2$ and $R\le\rho\le R+H$, combining these bounds
yields
\[
 \begin{split}
 \frac{|V_N(\rho;v)|}{N}\ll_\varepsilon R^\varepsilon\big(&N^{1/4}+R^{1/8}N^{1/16}\\
       &+\min\{R^{3/8}N^{-1/16},R^{97/336}N^{13/672}\}\big).
 \end{split}
\]
The minimum is at most the first term to the power $13/55$ times
the second to the power $42/55$. The powers of $N$ cancel, giving
\begin{equation}\label{eq:V-combined}
 \frac{|V_N(\rho;v)|}{N}\ll_\varepsilon R^\varepsilon
       \left(N^{1/4}+R^{1/8}N^{1/16}+R^{17/55}\right).
\end{equation}

\subsection{Summation of the dual series}
Use the smooth dyadic partition $w$ from Section~\ref{sec:arithmetic}
and put
\[
 E_N=2\sum_{n\ge1}\frac{r_3(n)}{\sqrt n}w(n/N)\mathcal I(\sqrt n)
 \qquad(N=1,2,4,\ldots).
\]
By \eqref{eq:radial-poisson}, the analytic error is $\sum_N E_N$.
We derive two bounds for each term and then sum them directly.

Put $F(r)=rf(r)$. Integration by parts gives
\begin{align}
 \mathcal I(\xi)
 &=\frac1{2\pi\xi}\int_R^{R+H}rf'(r)\cos(2\pi r\xi)\,dr\notag\\
 &\hspace{8mm}-\frac1{(2\pi\xi)^2}\int_R^{R+H}f'(r)\sin(2\pi r\xi)\,dr,
 \label{eq:IBP-first}\\
 \mathcal I(\xi)
 &=-\frac1{(2\pi\xi)^2}\int_R^{R+H}F''(r)\sin(2\pi r\xi)\,dr.
 \label{eq:IBP-second}
\end{align}
There is no contribution at the origin: $F(0)=0$ in the first
integration and $\sin(0)=0$ in the second. All the remaining derivatives
are supported in $[R,R+H]$. In particular,
\[
 \|rf'\|_1\ll R,\qquad \|f'\|_1\ll1,
 \qquad \|F''\|_1\ll R/H.
\]

For $a=1,3/2$, write $w_a(x)=x^{-a}w(x)$, extended by zero outside
$x>0$. Then
\[
 \sum_{n\ge1}\frac{r_3(n)}{n^a}w(n/N)e(r\sqrt n)
       =N^{-a}V_N(r;w_a).
\]
Substitution of \eqref{eq:IBP-first} into $E_N$ gives the exact identity
\[
 \begin{split}
 E_N={}&\frac1{\pi N}\int_R^{R+H}rf'(r)
                    \operatorname{Re}V_N(r;w_1)\,dr\\
 &-\frac1{2\pi^2N^{3/2}}\int_R^{R+H}f'(r)
                    \operatorname{Im}V_N(r;w_{3/2})\,dr.
 \end{split}
\]
Using \eqref{eq:IBP-second} instead gives
\[
 E_N=-\frac1{2\pi^2N^{3/2}}\int_R^{R+H}F''(r)
                    \operatorname{Im}V_N(r;w_{3/2})\,dr.
\]
The two resulting bounds combine to
\begin{equation}\label{eq:analytic-block}
 |E_N|\ll\frac RN\min\{1,(H\sqrt N)^{-1}\}
       \sup_{R\le\rho\le R+H}
            \max_{a\in\{1,3/2\}}|V_N(\rho;w_a)|.
\end{equation}

For $N\le R^2$, insert \eqref{eq:V-combined}. If $0<a<1/2$, summing
geometrically on either side of $N=H^{-2}$ gives
\[
 \sum_{N=1,2,4,\ldots}N^a\min\{1,(H\sqrt N)^{-1}\}
       \ll_a H^{-2a}.
\]
For $a=0$ and $N\le R^2$, the corresponding sum is $O(\log(2R))$.
Applying these statements with $a=1/4,1/16,0$ in
\eqref{eq:analytic-block}, we obtain
\begin{equation}\label{eq:analytic-main-range}
 \sum_{N\le R^2}|E_N|\ll_\varepsilon R^\varepsilon
       \left(RH^{-1/2}+R^{9/8}H^{-1/8}+R^{72/55}\right).
\end{equation}

For $N>R^2$, cancellation in $V_N$ is unnecessary. Two more
integrations by parts in \eqref{eq:IBP-second} have no boundary terms,
because $F''$ and its derivatives vanish at both ends of their
support. They give
\[
 |\mathcal I(\xi)|\ll\xi^{-4}\|F^{(4)}\|_1
                 \ll RH^{-3}\xi^{-4}.
\]
The elementary bound
\[
 \sum_{1\le n\le2N}r_3(n)\le(2\sqrt{2N}+1)^3\ll N^{3/2}
\]
therefore implies $|E_N|\ll RH^{-3}N^{-1}$. Summing over the dyadic
scales $N>R^2$ gives
\begin{equation}\label{eq:analytic-tail}
 \sum_{N>R^2}|E_N|\ll R^{-1}H^{-3}.
\end{equation}
This also proves the absolute convergence used in decomposing the
series.

Finally, for $H=R^{-13/21}$ the terms in
\eqref{eq:analytic-main-range} and \eqref{eq:analytic-tail} have sizes
\[
 R^{55/42},\qquad R^{101/84},\qquad R^{72/55},\qquad R^{6/7}.
\]
The last three exponents are smaller than $55/42$. Substituting in
\eqref{eq:radial-poisson} proves Proposition~\ref{prop:analytic}.

\section{Completion of the proof of Theorem \ref{thm:main}}\label{sec:complete}
By \eqref{eq:decomposition} and Propositions~\ref{prop:arithmetic}
and~\ref{prop:analytic},
\[
 \begin{split}
 S(R)&=4\pi\left(\int_0^\infty r^2f(r)\,dr
                    -\int_R^{R+H}r^2f(r)\,dr\right)
                    +O_\varepsilon(R^{55/42+\varepsilon})\\
     &=\frac{4\pi}{3}R^3+O_\varepsilon(R^{55/42+\varepsilon}),
 \end{split}
\]
because $f=1$ on $[0,R]$ and $f=0$ on $[R+H,\infty)$.
This proves Theorem~\ref{thm:main}.

\section{Proof of Theorem \ref{thm:qrh}}\label{sec:qrh-refinement}
We prove the stronger estimate with error exponent
\[
 \Theta=\frac{2274811}{1743872}.
\]
We now assume the zero-free half-plane $\operatorname{Re}s>7/8$ for
Dirichlet $L$-functions from \cite{OpenAIQRH}. The consequence that we
need is a shorter Dirichlet approximation to $L(1,\chi)$. More precisely,
for a nonprincipal quadratic Dirichlet character $\chi$ modulo $q$, the
argument in \cite{OpenAIQRH} gives
\begin{equation}\label{eq:qrh-truncation}
 L(1,\chi)=\sum_{m\ge1}\frac{\chi(m)}m U(m/Y)
 +O_\varepsilon\bigl(q^\varepsilon Y^{-1/8}\bigr),
\end{equation}
where $U$ is a fixed smooth cutoff equal to one on $[0,1]$ and zero on
$[2,\infty)$. We apply this with
\[
 Y\asymp q^\alpha,
 \qquad \alpha=\frac5{13}.
\]
The same estimate applies to the characters $\chi_n$ occurring in
\eqref{eq:class-number}, taking $q=4n$.
For the proof of Theorem~\ref{thm:qrh}, we repeat the smoothing
decomposition with
\begin{equation}\label{eq:H-qrh}
 H=R^{-3/5}.
\end{equation}
All the identities in Section~\ref{sec:decomposition} remain unchanged.
We first reconsider the arithmetic boundary term. For fixed $d$, put
$Q=R/d$ and take $Y_d=(16Q^2)^\alpha\asymp Q^{2\alpha}$.
In the nonsquare contribution and its dyadic decomposition, we replace
$U(m/Z_d)$ by $U(m/Y_d)$. The notation $E(R)$, $E_\nu(d;L,J)$, and
$W$ below refers to the corresponding quantities with this replacement.
Since $n\asymp Q^2$ on the interval in \eqref{eq:boundary-class},
\eqref{eq:qrh-truncation}, after multiplication by $\sqrt n$, gives an
error
\[
 O_\varepsilon\bigl(R^\varepsilon Q^{1-\alpha/4}\bigr)
\]
for each pair $(d,n)$. Using the estimate $O(1+RH/d^2)$ for each $d$
from \eqref{eq:pair-count}, the total truncation error is
\begin{equation}\label{eq:qrh-total-truncation}
 \begin{split}
 &\ll_\varepsilon R^\varepsilon
   \sum_{d\le R+H}\left(1+\frac{RH}{d^2}\right)
                  \left(\frac Rd\right)^{1-\alpha/4}
 \ll_\varepsilon R^{1+\varepsilon}
                  +R^{2-\alpha/4+\varepsilon}H.
 \end{split}
\end{equation}
Here $0<1-\alpha/4<1$. For the chosen $H$ and $\alpha$,
\[
 2-\frac\alpha4-\frac35=\frac{339}{260}<\Theta.
\]
The square terms can again be completed to the full square sum. Indeed,
the omitted square tail contributes at most
\[
 \sqrt n\sum_{a>Y_d^{1/2}}a^{-2}\ll Q^{1-\alpha}
\]
for each pair $(d,n)$, and hence in total
\begin{equation}\label{eq:qrh-square-tail}
 \ll R+R^{2-\alpha}H\ll R^{66/65}.
\end{equation}
Thus the evaluation of $M(R)$ in \eqref{eq:squares-evaluated} is
unchanged.\\
The shortened Dirichlet approximation gives the range
\begin{equation}\label{eq:qrh-J-range}
 J\ll Q^{2\alpha}.
\end{equation}
The normalization \eqref{eq:normalization} and the uniform derivative
bounds for the weights remain valid. In particular,
$W(x)=x^{-1}w(x)U(Jx/Y_d)$ has uniformly bounded derivatives because
$J\ll Y_d$.
We keep both Proposition~\ref{prop:fourth} and the Petrow--Young
estimate \eqref{eq:PY-bound}. For nonnegative $A,B,C$,
\[
 \min(A+B,C)\le A+\min(B,C)\le A+B^{4/5}C^{1/5}.
\]
Apply this with
\[
 A=(JL)^{1/4},\qquad B=J^{1/2}L^{-1/4},\qquad
 C=Q^{1/3}LJ^{-1/2}.
\]
Since
\[
 B^{4/5}C^{1/5}=Q^{1/15}J^{3/10},
\]
we obtain
\begin{equation}\label{eq:qrh-hybrid-bound}
 S_\nu(Q^2;L,J)\ll_\varepsilon R^\varepsilon
 \left((JL)^{1/4}+Q^{1/15}J^{3/10}\right).
\end{equation}
Using \eqref{eq:normalization} and \eqref{eq:qrh-J-range} gives
\[
 E_\nu(d;L,J)\ll_\varepsilon R^\varepsilon
 \left(Q^{1+\alpha/2}L^{1/4}
       +Q^{16/15+3\alpha/5}\right).
\]
We sum over the dyadic parameters as in \eqref{eq:E-blocks}, using
$L\ll RH/d^2$. The first term is summed geometrically in $L$.
The second, being independent of $L$, costs only a logarithmic factor.
The dyadic $J$-sum also costs at most a logarithmic factor. The sums
in $d$ are bounded by
\[
 \sum_{d\ge1}d^{-3/2-\alpha/2}
 \quad\text{and}\quad
 \sum_{d\ge1}d^{-16/15-3\alpha/5},
\]
both of which converge. Absorbing the logarithmic factors into
$R^\varepsilon$, we obtain
\begin{equation}\label{eq:qrh-arithmetic-error}
 E(R)\ll_\varepsilon R^\varepsilon\left(
 R+R^{5/4+\alpha/2}H^{1/4}
   +R^{16/15+3\alpha/5}\right).
\end{equation}
For $H=R^{-3/5}$ and $\alpha=5/13$, the two nontrivial exponents are
\[
 \frac54+\frac\alpha2-\frac3{20}=\frac{84}{65},
 \qquad
 \frac{16}{15}+\frac{3\alpha}{5}=\frac{253}{195}.
\]
Both are smaller than $339/260$. Together with
\eqref{eq:qrh-total-truncation}, \eqref{eq:qrh-square-tail}, and
\eqref{eq:squares-evaluated}, this proves
\begin{equation}\label{eq:qrh-boundary}
 B^{(H)}(R)=4\pi\int_R^{R+H}r^2f(r)\,dr
 +O_\varepsilon(R^{339/260+\varepsilon}).
\end{equation}
It remains to improve the analytic estimate. For every sufficiently
small $\eta>0$, Trudgian and Yang \cite[Lemma~1.1]{TrudgianYang}
prove that $P_1+(\eta,\eta)$ and $P_2+(\eta,\eta)$ are exponent
pairs, where
\[
 P_1=\left(\frac{4742}{38463},\frac{35731}{51284}\right),
 \qquad
 P_2=\left(\frac{18}{199},\frac{593}{796}\right).
\]
The van der Corput $B$-process is
$B(k,\ell)=(\ell-1/2,k+1/2)$, so
\[
 B(P_1)=\left(\frac{10089}{51284},\frac{47947}{76926}\right),
 \qquad
 B(P_2)=\left(\frac{195}{796},\frac{235}{398}\right).
\]
Put
\[
 u=\frac{1961613}{3487744},
 \qquad 1-u=\frac{1526131}{3487744}.
\]
By the $B$-process and convexity of exponent pairs,
$(\kappa_*+\eta,\lambda_*+\eta)$ is an exponent pair, where
\begin{equation}\label{eq:qrh-exponent-pair}
 \begin{split}
 (\kappa_*,\lambda_*)&=uB(P_1)+(1-u)B(P_2)
 =\left(\frac{94971}{435968},\frac{530939}{871936}\right).
 \end{split}
\end{equation}
The relevant identities are
\begin{equation}\label{eq:qrh-pair-identities}
 2\lambda_*-\kappa_*=1,
 \qquad
 p:=\frac{1+\kappa_*}{4}
   =\frac{530939}{1743872}=\Theta-1.
\end{equation}
In particular,
\[
 \frac78+\frac{\lambda_*+\eta}{4}
          -\frac{\kappa_*+\eta}{8}
 =1+\frac\eta8.
\]
Applying Lemma~\ref{lem:exponent-pair} with this exponent pair gives,
for $\rho\ge2$ and $1\le N\le\rho^2$,
\[
 V_N(\rho;v)\ll_{\varepsilon_1,\eta}(\rho N)^{\varepsilon_1}
 \left(N^{5/4}+\rho^{p+\eta/4}N^{1+\eta/8}\right).
\]
The differentiated error estimate \eqref{eq:phase-derivatives} supplies
the finitely many derivative bounds required for this fixed pair,
exactly as in the proof of Lemma~\ref{lem:exponent-pair}.
Choose $\varepsilon_1,\eta>0$ sufficiently small in terms of
$\varepsilon$. Since $\rho\asymp R$ and $N\le R^2$, the extra
factors $\rho^{\eta/4}N^{\eta/8}$ are absorbed into $R^\varepsilon$.
Thus, uniformly for $R\le\rho\le R+H$ and $1\le N\le R^2$,
\begin{equation}\label{eq:qrh-V-bound}
 \frac{|V_N(\rho;v)|}{N}\ll_\varepsilon R^\varepsilon
 \left(N^{1/4}+R^p\right).
\end{equation}
Insert \eqref{eq:qrh-V-bound} into \eqref{eq:analytic-block}.
The dyadic summation used in Section~\ref{sec:analytic} gives
\[
 \sum_{N=1,2,4,\ldots}N^{1/4}
       \min\{1,(H\sqrt N)^{-1}\}\ll H^{-1/2},
\]
while
\[
 \sum_{\substack{N=1,2,4,\ldots\\N\le R^2}}
       \min\{1,(H\sqrt N)^{-1}\}\ll\log(2R).
\]
Consequently, after absorbing the logarithmic factor into
$R^\varepsilon$,
\begin{equation}\label{eq:qrh-analytic-main-range}
 \sum_{N\le R^2}|E_N|\ll_\varepsilon R^\varepsilon
 \left(RH^{-1/2}+R^{1+p}\right).
\end{equation}
The tail estimate \eqref{eq:analytic-tail} is unchanged. Hence
\[
 \sum_N|E_N|\ll_\varepsilon R^\varepsilon
 \left(RH^{-1/2}+R^\Theta+R^{-1}H^{-3}\right).
\]
With $H=R^{-3/5}$, the first and third terms are respectively
$R^{13/10}$ and $R^{4/5}$, both smaller than $R^\Theta$.
It follows from \eqref{eq:radial-poisson} that
\begin{equation}\label{eq:qrh-analytic}
 S^{(H)}(R)=4\pi\int_0^\infty r^2f(r)\,dr
 +O_\varepsilon(R^{\Theta+\varepsilon}).
\end{equation}
Finally, $339/260<\Theta$, so \eqref{eq:decomposition},
\eqref{eq:qrh-boundary}, and \eqref{eq:qrh-analytic} give
\[
 \begin{split}
 S(R)&=4\pi\left(\int_0^\infty r^2f(r)\,dr
                    -\int_R^{R+H}r^2f(r)\,dr\right)
                    +O_\varepsilon(R^{\Theta+\varepsilon})\\
     &=\frac{4\pi}{3}R^3
                    +O_\varepsilon(R^{2274811/1743872+\varepsilon}).
 \end{split}
\]
Here the cancellation of the two main integrals is exact, as in
Section~\ref{sec:complete}. This proves Theorem~\ref{thm:qrh}.

\section*{Acknowledgments}
The authors would like to thank D. R. Heath-Brown and A. Yang for their
insightful comments and advice throughout the preparation of this article. The first author is grateful to Max Planck Institute for Mathematics in Bonn for its hospitality and financial support. The authors thank ChatGPT for assistance in editing and proof reading the manuscript.

\end{document}